\documentclass[11pt]{article}
\usepackage{amsmath,amssymb}
\usepackage{color}
\usepackage{mathtools}
\usepackage[T1]{fontenc}

\newtheorem{theorem}{Theorem}
\newtheorem{lemma}{Lemma}

\newtheorem{remark}{Remark}

\def\qed{{\ \hfill\hbox{\vrule width1.0ex height1.0ex}\parfillskip 0pt
	}}
	\newenvironment{proof}{\noindent{\bf Proof:}}{\qed}
\begin{document}
\title{A characterization of normality via convex likelihood ratios}
\author{Royi Jacobovic  \thanks{Department of Statistics and data science; The Hebrew University of Jerusalem; Jerusalem 9190501; Israel.
{\tt royi.jacobovic@mail.huji.ac.il, offer.kella@gmail.com}} \thanks{Supported by the GIF Grant 1489-304.6/2019} \and Offer Kella\footnotemark[1] \thanks{Supported by grant No. 1647/17 from the Israel Science Foundation and the Vigevani Chair in Statistics.}}

\date{March 3, 2022}
		\maketitle
		\begin{abstract}
			This work includes a new characterization of the multivariate normal distribution. In particular, it is shown that a positive density function $f$ is Gaussian if and only if the $f(x+ y)/f(x)$ is convex in $x$ for every $y$. This result has implications to recent research regarding inadmissibility of a test studied by Moran~(1973).
		\end{abstract}


\bigskip
\noindent {\bf AMS Subject Classification (MSC2020):} 62H05, 62E10.

\section{Introduction}
There are many characterizations of the normal distribution. For some surveys see, {\em e.g.,} \cite{hamedani1992,kagan1973,patel1996,tong2012} and more recently, \cite{Azzalini20007,e2016,hgm19,novak2007,v14}.  The present study is about a new characterization, which is that a positive (possibly multivariate) density $f$ is Gaussian if and only if $f(x+y)/f(x)$ is convex in $x$ for every $y$. Note that the latter is equivalent to the convexity of the likelihood ratio $f(x-\theta_1)/f(x-\theta_0)$ in $x$ for every $\theta_0\not=\theta_1$ where $\theta_1,\theta_2$ are (possibly multivariate) location parameters. A characterization of the one parameter exponential family via monotonicity of (more general) likelihood ratios may be found in \cite{bp63}.

The motivation for this follows from a statistical hypothesis problem. Assume that one observes an i.i.d. sample. The distribution is known up to a location parameter $y$ and the need arises to test whether $y$ is the null vector or not. In \cite{m73} Moran suggested a test which is based on a single split of the data. The rationale is to use the first part of the data in order to estimate in what direction some estimator of
$y$ diverges from the null. Then,  the second part is applied in order to test whether the true value of
$y$ diverges from the null in this direction. Theorem 4 in  \cite{j21} asserts that once the observations are normally distributed, then this test is inadmissible. A careful reading of the proof of this theorem yields that the same conclusion can be made for every density function $f(\cdot)$ such that $f(x+ y)/f(x) $ is convex in $x$ for every $y$. This issue, has been the initial motivation to  characterize the space of density functions satisfying this property, which as announced can unfortunately only be Gaussian. 

One potential direction for further research is the possible application of the current characterization to devising new normality tests. For various examples of normality tests see, {\em e.g.,} \cite{henze2002,jelito2021,Khatun2021} and references therein.

\section{Main results}\label{sec: result}
We begin with the following.
\begin{lemma}\label{l1} $g:\mathbb{R}^n\to\mathbb{R}$ is convex, strictly positive and $g(x+y)/g(x)$ is convex in $x$ for every $y$ if and only if 
for some symmetric positive semidefinite $A\in\mathbb{R}^{n\times n}$, $b\in\mathbb{R}^n$ and $c\in \mathbb{R}$,
\begin{equation}\label{eq1}
g(x)=\exp\left(\frac{1}{2}x^TAx+b^Tx+c\right)\ .
\end{equation}
\end{lemma}

\begin{proof}
If (\ref{eq1}) holds with a positive semidefinite $A$, then $g$ is log-convex. Similarly $g(x+y)/g(x)$ is log-linear in $x$ for every $y$, hence also log-convex. log-convex functions are convex.

To show the converse denote
\begin{equation}
g(x;y)=\lim_{t\downarrow 0}\frac{g(x+ty)-g(x)}{t}
\end{equation}
Since $g(x+ty)$ is convex in $t\in\mathbb{R}$, this limit exists and is finite for every $x,y$. Since $g(x+ty)/g(x)$ is convex in $x$ for every $t,y$ then so is
\begin{equation}
\frac{g(x+ty)-g(x)}{g(x)}=\frac{g(x+ty)}{g(x)}-1
\end{equation}
and thus, letting $t\downarrow 0$ implies that $g(x;y)/g(x)$ is convex and therefore also continuous in $x$ for every $y$. Since $g$ is convex then it is also continuous and thus $g(x;y)$ is continuous in $x$ for every $y$.

By Theorem~25.5 of~\cite{r97}, $g$ is differentiable on a dense subset $D$ of $\mathbb{R}^n$. Therefore, for every $x\in D$ we have that $g(x;y)=\nabla g(x)^Ty$. Since $g(x;y)$ is continuous in $x$ then for every sequence $x_n\in D$ such that $x_n\to x$ we have that 
\begin{equation}
g(x;y)=\lim_{n\to\infty}\nabla g(x_n)^Ty
\end{equation}
so that necessarily $g(x;y)$ is linear in $y$. Theorem~25.2 of~\cite{r97} implies that $g$ is (necessarily continuously) differentiable on $\mathbb{R}^n$. Moreover $\nabla g(x)^Ty/g(x)$ is convex in $x$ for every $y$. Therefore, so is $-\nabla g(x)^Ty/g(x)=\nabla g(x)^T(-y)/g(x)$ which implies that $\nabla g(x)^Ty/g(x)$ is affine in $x$ for every $y$. In particular, taking $y_i=1$ and $y_j=0$ for $j\not=i$ implies that $\frac{\partial g(x)}{\partial x_i}/g(x)=a_i^Tx+b_i$ for some $a_i\in \mathbb{R}^n$ and $b_i\in\mathbb{R}$. Therefore, if $A$ is a matrix with rows $a_i$ and $b=(b_1,\ldots,b_n)^T$ we have that
\begin{equation}
\nabla \log(g(x))=\frac{\nabla g(x)}{g(x)}=Ax+b\,.
\end{equation}
Hence,
\begin{align}
\log(g(x))-\log(g(0))&=\int_0^1\frac{{\rm d}\log(g(tx))}{{\rm d}t}{\rm d}t=\int_0^1 (A(tx)+b)^Tx{\rm d}t\nonumber \\
&=\frac{1}{2}x^TA^Tx+b^Tx=\frac{1}{2}x^TAx+b^Tx\,.
\end{align}
Without loss of generality $A$ is symmetric. If not then we replace it with $(A+A^T)/2$ and the right hand side remains the same. Taking $c=\log(f(0))$ implies~(\ref{eq1}). It remains to show that $A$ is positive semidefinite. If it is not then there exists $\lambda<0$ and $y\not=0$ such that $Ay=\lambda y$. Taking $x=-b/\lambda$ implies that $y^T(Ax+b)=0$ and thus
we have that
\begin{align}
y^T\nabla^2 g(x)y&=y^T\left(g(x)(A+(Ax+b)(Ax+b)^T)\right)y\\ &=g(x)y^TAy=g(x)\lambda \|y\|^2<0\nonumber
\end{align}
contradicting the convexity of $g$.
\end{proof}

The following is our main result.
\begin{theorem}\label{t1}
$f:\mathbb{R}^n\to\mathbb{R}$ is positive, Borel and integrates to $1$ (a positive density function). The following are equivalent.
\begin{description}
\item{\rm(i)} $f(x+y)/f(x)$ is convex in $x$ for every $y$.
\item{\rm(ii)} $f(x+y)/f(x)$ is log-convex in $x$ for every $y$.
\item{\rm(iii)} $f(x+y)/f(x)$ is log-concave in $x$ for every $y$.
\item{\rm(iv)} $f$ is a multivariate normal density (necessarily with positive definite covariance matrix). 
\end{description}
\end{theorem}

\begin{proof}
\begin{description}
\item{(iv)$\Rightarrow$(iii)} $f(x+y)/f(x)$ is log-linear in $x$ for every $y$.

\item{(iii)$\Rightarrow$(ii)} If $h(x,y)=f(x+y)/f(x)$ is log-concave in $x$ for every $y$, then so is $h(x+y,-y)=f(x)/f(x+y)$. Thus $f(x+y)/f(x)$ is log-convex in $x$ for every $y$.

\item{(ii)$\Rightarrow$(i)} Any log-convex function is convex.

\item{(i)$\Rightarrow$(iv)} Since $f$ is a density then integrating $f(x+y)/f(x)$ with respect to $y$ implies that $g(x)=1/f(x)$ is convex. Also, if $h(x,y)=f(x+y)/f(x)$ is convex in $x$ for every $y$ then so is $h(x+y,-y)=f(x)/f(x+y)$ and thus $g(x+y)/g(x)$ is convex in $x$ for each $y$. By Lemma~\ref{l1} we must have that for some symmetric positive semidefinite $A\in\mathbb{R}^{n\times n}$, $b\in \mathbb{R}^n$ and $c\in \mathbb{R}$ we have that
\begin{equation}
f(x)=1/g(x)=\exp\left(-\frac{1}{2}x^TAx-b^Tx-c\right)\,.
\end{equation}
Let $O\Lambda O^T$ be the spectral decomposition of $A$ with $O^TO=OO^T=I$ and $\Lambda=\text{diag}(\lambda_1,\ldots,\lambda_n)$ where $\lambda_i\ge 0$ for all $i$. Then, $f(Ox)|O|=f(Ox)$ is also a density satisfying
\begin{equation}
f(Ox)=\exp\left(-\frac{1}{2}x^T\Lambda x-(O^Tb)^Tx-c\right)\,.
\end{equation}
If $\lambda_i=0$ for any $i$ then the integral with respect to $x_i$ is infinite and thus $f(Ox)$ cannot be a density. Therefore $\lambda_i>0$ for all $i$ and thus $A$ is positive definite. This implies that $f$ is a multivariate normal density with positive definite covariance matrix $\Sigma=A^{-1}$ and mean vector $\mu=-A^{-1}b$.
\end{description}
\end{proof}

\section{Some discussion}\label{sec: counter-examples}
It would be interesting to check whether the conditions of Theorem \ref{t1} can be weakened. We do this with a series of remarks.

\begin{remark}{\rm
Given the equivalence of (ii) and (iii) of Theorem~\ref{t1} a natural question is why did we not consider concavity of $f(x+y)/f(x)$ in $x$ for all $y$? This is because there is no density satisfying this. The reason is that if $f(x+y)/f(x)$ is concave in $x$ for each $y$, then by integrating with respect to $y$, $1/f(x)$ is concave and positive and thus must be a constant (since it is bounded below by zero). Therefore $f(x)$ must be a positive constant which is impossible since $f$ is a density.
}\end{remark}

\begin{remark}\label{diff}{\rm
In Lemma~\ref{l1} if instead of convexity of $g$ we assume that $g$ is differentiable (all else being the same), then we would instantly arrive at the conclusion that $\nabla g(x)^Ty/g(x)$ is convex (and thus concave and therefore affine) in $x$ for each $y$. We note that in this case it would even suffice to assume that $g(x+y)/g(x)$ is convex in $x$ for all $y$ in a neigborhood of zero. In this case one would have that $A$ that appears in this lemma need not be positive semidefinite. Although this is true, this would not help us in the proof of Theorem~\ref{t1} for which Lemma~\ref{l1} was needed. Note that only the convexity of $f(x+y)/f(x)$ in $x$ for each $y$ is assumed in~(i) of Theorem~\ref{t1}. Nothing about the properties of $f$ (such as differentiability/convexity/concavity) other than it being a positive density is assumed. 

We note that if we assume that $f$ is differentiable, then the condition that $f(x+y)/f(x)$ is convex in $x$ for ``all $y$'' can be weakened to ``for all $y$ in a neighborhood of the origin''. Our view is that assuming differentiability on top of the convexity in $x$ of $f(x+y)/f(x)$, rather than inferring it, substantially weakens the result.
}\end{remark}

\begin{remark}{\rm
Can the convexity of $f(x+y)/f(x)$ in $x$ for each $y$ be replaced with (the weaker) quasi-convexity? The answer is no. For example take $f(x)=e^{-|x|}/2$  ($x\in\mathbb{R}$) which is, of course, not Gaussian. It is easy to check that with $y^+=\max(y,0)$ and $y^-=-\min(y,0)$,

\begin{equation}
\log\frac{f(x+y)}{f(x)}=\begin{cases}
y&x\in(-\infty,-y^+]\\
{-y-2x}&x\in (-y^+,0]\\
{y+2x}&x\in (0,-y^-]\\
{-y}&x\in[y^-,\infty)\,.
\end{cases}
\end{equation}
Therefore, noting that either $(-y^+,0]$ or $(0,-y^-]$ is empty, $\log (f(x+y)/f(x))$ and thus $f(x+y)/f(x)$ is monotone, hence quasi-convex in $x$ for each $y$.
}\end{remark}

\begin{remark}{\rm
In Remark~\ref{diff} we mentioned that when $f$ is differetiable, then convexty of $f(x+y)/f(x)$ in $x$ for all $y$ in a neighborhood of zero (rather than for all $y$) implies that $f$ is Gaussian. Are there non-Gaussian differentiable densities for which $f(x+y)/f(x)$ is convex in $x$ for all $y\not\in(-\epsilon,\epsilon)$ for some $\epsilon>0$? 
The answer is yes. 

Let $f(x)=ce^{-x^4}$, where $c$ is a normalizing constant. Then it is straightforward to check that the second derivative of $h(x,y)=f(x+y)/f(x)$ with respect to $x$ is given by
\begin{align}
h_{xx}(x,y)=h(x,y)\left(-12(2x+y)y+y^2(3(2x+y)^2+y^2)^2\right)\,.
\end{align}
Since $h_{xx}(y)/y\to -24 x$ as $y\to 0$ we clearly have that for any $x\not=0$ there are $y$ sufficiently close to zero for which $h_{xx}(x,y)<0$ and thus $h(x,y)$ is not convex in $x$ for such $y$. This is, of course, expected for any (differentiable) non-Gaussian density.

On the other hand, since $2ab\le a^2+b^2$, 
\begin{equation}
2(2x+y)y\le (2x+y)^2+y^2\le 3(2x+y)^2+y^2
\end{equation}
and thus for all $y$ such that $y^2\ge 6$ we clearly have that
\begin{equation}
12(2x+y)y\le y^2( 3(2x+y)^2+y^2)\le y^2 (3(2x+y)^2+y^2)^2
\end{equation}
and thus, with $\epsilon=\sqrt{6}$, $f(x,y)/f(x)$ is convex in $x$ for all $y\not\in \left(-\epsilon,\epsilon\right)$.
}\end{remark}

\end{document}